\documentclass{article}
\usepackage{amsmath}
\usepackage{harvard}

\newtheorem{theorem}{Theorem}

\newenvironment{proof}[1][Proof]{\noindent\textbf{#1.} }{\ \rule{0.5em}{0.5em}}

\input{tcilatex}

\begin{document}

\title{A Note On The Spectral Norms of The Circulant Matrices Connected
Integer Number Sequences}
\author{Durmu\c{s} Bozkurt\thanks{%
e-mail: dbozkurt@selcuk.edu.tr} \and Department of Mathematics, Science
Faculty of Sel\c{c}uk University}
\maketitle

\begin{abstract}
In this paper, we compute the spectral norms of the matrices related with
integer squences and we give some example related with Fibonacci, Lucas,
Pell and Perrin numbers.
\end{abstract}


MS Classification Number: 15A60, 15F35, 15B36, 15B57

\section{Introduction}

\bigskip In [1], the upper and lower bounds for the spectral norms of the
matrices with Fibonacci and Lucas numbers are obtained by S. Solak. In [2],
Ipek obtained as $\left\Vert A\right\Vert _{2}=F_{n+1}-1$ and $\left\Vert
B\right\Vert _{2}=F_{n+2}+F_{n}-1$ the spectral norms of the matrices by
Solak defined in [1] where $F_{i}$ is the $i$th Fibonacci number.

Let $A$ be any $n\times n$ complex matrix. The well known the spectral norm
of the matrix $A$ is%
\begin{equation*}
\left\Vert A\right\Vert _{2}=\sqrt{\max_{1\leq i\leq n}\left\vert \lambda
_{i}(A^{H}A)\right\vert }
\end{equation*}%
where $\lambda _{i}(A^{H}A)$ is eigenvalue of $A^{H}A$\ and $A^{H}$ is
conjugate transpose of the matrix $A$.

\bigskip By a circulant matrix of order $n$ is meant a square matrix of the
form%
\begin{equation*}
C=circ(c_{0},c_{1},\ldots ,c_{n-1})=\left[ 
\begin{array}{ccccc}
c_{0} & c_{1} & c_{2} & \ldots & c_{n-1} \\ 
c_{n-1} & c_{0} & c_{1} & \ldots & c_{n-2} \\ 
c_{n-2} & c_{n-1} & c_{0} & \ldots & c_{n-3} \\ 
\vdots & \vdots & \vdots & \ddots & \vdots \\ 
c_{1} & c_{2} & c_{3} & \ldots & c_{0}%
\end{array}%
\right] .[3]
\end{equation*}

Now we define our matrix. $(x_{n})$ is any positive integer numbers squence
and $x_{i}$ is $i$th the component of the sequence $(x_{n})$ for $%
i=0,1,2,\ldots .$ Let matrix $A_{x}$ be following form:%
\begin{equation}
A_{x}=circ(x_{0},x_{1},\ldots ,x_{n-1}).  \tag{1}  \label{1}
\end{equation}

The main objective of this paper is to obtain the spectral norms of the
matrices $A_{x}$ in (\ref{1}).

\section{Main Result}

\begin{theorem}
Let the matrix $A_{x}$ be as in (\ref{1}). Then%
\begin{equation*}
\left\Vert A_{x}\right\Vert _{2}=\dsum\limits_{i=0}^{n-1}x_{i}.
\end{equation*}%
where $\left\Vert .\right\Vert _{2}$ is the spectral norm and $x_{j}$ \ are $%
j$th components of the sequence $(x_{n})$.
\end{theorem}

\begin{proof}
Since the circulant matrices are normal, the spectral norm of the circulant $%
A_{x}$ is equal to its spectral radius. Also the circulant $A_{x}$ is
irreducible and its entries are nonnegative. Therefore the spectral radius
of the matrix $A_{x}$ is the same its Perron root.

Let $y$ be a vector with all the components $1$. Then%
\begin{equation*}
A_{x}y=\left( \dsum\limits_{i=0}^{n-1}x_{i}\right) y.
\end{equation*}%
Since $\tsum\nolimits_{i=0}^{n-1}x_{i}$ is an eigenvalues of $A_{x}$
corresponding a positive eigenvector, it must be the Perron root of the
matrix $A_{x}$. Then we have%
\begin{equation*}
\left\Vert A_{x}\right\Vert _{2}=\dsum\limits_{i=0}^{n-1}x_{i}.
\end{equation*}%
The proofs are completed.\bigskip
\end{proof}

\section{Numerical Consideration}

\bigskip The well known $F_{n},L_{n},P_{n}$ and $R_{n}$ are $n$th Fibonacci,
Lucas, Pell and Perrin numbers with recurence relations 
\begin{equation*}
F_{n}=\left\{ 
\begin{array}{lr}
0 & \text{if }n=0 \\ 
1 & \text{if }n=1 \\ 
F_{n}=F_{n-1}+F_{n-2} & \text{if }n>1%
\end{array}%
,\right. L_{n}=\left\{ 
\begin{array}{lr}
2 & \text{if }n=0 \\ 
1 & \text{if }n=1 \\ 
L_{n}=L_{n-1}+L_{n-2} & \text{if }n>1%
\end{array}%
\right.
\end{equation*}%
\begin{equation*}
P_{n}=\left\{ 
\begin{array}{lr}
0 & \text{if }n=0 \\ 
1 & \text{if }n=1 \\ 
P_{n}=2P_{n-1}+P_{n-2} & \text{if }n>1%
\end{array}%
,\right. R_{n}=\left\{ 
\begin{array}{rr}
3 & \text{if }n=0 \\ 
0 & \text{if }n=1 \\ 
2 & \text{if }n=2 \\ 
R_{n}=R_{n-2}+R_{n-3} & \text{if }n>2%
\end{array}%
\right. ,
\end{equation*}%
respectively.

If the sequence $(x_{n})$ is the Fibonacci and Lucas sequences then the
following results are obtained:%
\begin{equation*}
\left\Vert A_{F}\right\Vert _{2}=\dsum\limits_{i=0}^{n-1}F_{i}=F_{n+1}-1\ [2]
\end{equation*}%
and%
\begin{equation*}
\left\Vert A_{L}\right\Vert
_{2}=\dsum\limits_{i=0}^{n-1}L_{i}=F_{n+2}+F_{n}-1\ [2]
\end{equation*}

If the sequence $(x_{n})$ is Pell and Perrin sequences then we have%
\begin{equation*}
\left\Vert A_{P}\right\Vert _{2}=\dsum\limits_{i=0}^{n-1}P_{i}=\frac{1}{2}%
(P_{n}+P_{n-1}-1)
\end{equation*}

and%
\begin{equation*}
\left\Vert A_{R}\right\Vert _{2}=\dsum\limits_{i=0}^{n-1}R_{i}=R_{n+4}-1.
\end{equation*}

\end{document}